\documentclass{article}
\usepackage{latexsym}
\usepackage{amssymb}
\newtheorem{theorem}{Theorem}
\newtheorem{lemma}[theorem]{Lemma}
\newtheorem{corollary}[theorem]{Corollary}
\newtheorem{example}[theorem]{Example}
\newtheorem{technique}[theorem]{TECHNIQUE}
\newtheorem{proposition}[theorem]{Proposition}

\newtheorem{preproof}{{\bf Proof.}}
\renewcommand{\thepreproof}{}
\newenvironment{proof}[1]{\begin{preproof}{\rm
               #1}\hfill{\rule[-0.5mm]{2mm}{2mm}}}{\end{preproof}}

\def\mthree#1#2#3{\raise .8ex\hbox{%         % You can change 1
     ${#1_{{\displaystyle #2}_{\displaystyle #3}}}$}}

\def\mfour#1#2#3#4{\raise .02ex\hbox{%
    $#1_{\displaystyle #2}#3_{\displaystyle #4}$}}

\def\mfour#1#2#3#4{\hbox{%         % You can change 1
     ${#1_{{\displaystyle #2}}#3_{{\displaystyle #4}}}$}}

\def\mfour#1#2#3#4{\raise 1ex\hbox{${#1_{{\displaystyle #2}_{{\displaystyle #3}_{\displaystyle #4}}}}$}}

\def\m#1#2#3{\raise .8ex\hbox{%         % You can change 1
     ${#1_{{\displaystyle #2}_{\displaystyle #3}}}$}}

\def\x#1{\raise 0.5ex\hbox{%         % You can change 0.5
    ${#1}$}}
\def\n#1{\vbox to 3mm{\vspace{0mm}\vfill \hbox to 6.5mm{\hfill
             $#1$\hfill} \vfill }}

\def\m#1#2#3{\raise .8ex\hbox{%         % You can change 1
     ${#1_{{\displaystyle #2}_{\displaystyle #3}}}$}}

\def\b#1{\vbox to 5mm{\vspace{0mm}\vfill \hbox to 13.1mm{\hfill
             $#1$\hfill} \vfill }}

\def\r#1{\vbox to 8mm{\vspace{0mm}\vfill \hbox to 8mm{\hfill
             $#1$\hfill} \vfill }}
\def\r#1#2{\raise 0.2ex\hbox{%         % You can change 0.2
    ${#1_{\displaystyle #2}}$}}

\def\arraystretch{1.3}                 % You can change 1.8

\title{\bf The Four-way Intersection Problem for Latin Squares}
\author{
{\Large P. Adams${}^1$, E. S. Mahmoodian${}^2$\footnote{This research was partially supported by a grant from the INSF.},} \\
{\Large H. Minooei${}^2$, 
 M. Mohammadi Nevisi${}^2$} \\
$^1$ {\small Department of Mathematics, The University of Queensland,}\\
{\small Brisbane 4072, AUSTRALIA}\\
$^2$ {\small Department of Mathematical Sciences, Sharif University of Technology,}\\
{\small P.O. Box 11155--9415, Tehran, I.R. IRAN}\\
}

\date{}
\begin{document}
\maketitle

\begin{abstract}
For $\mu$ given latin squares of order $n$, they have {\sf $k$ intersection}
when they have $k$ identical cells and $n^2-k$ cells with mutually
different entries. For each $n\geq 1$ the set of integers $k$
such that there exist $\mu$ latin squares of order $n$ with $k$
intersection is denoted by $I^{\mu}[n]$. In a paper by P. Adams
et al. (2002), $I^3[n]$ is determined completely.
In this paper we completely determine $I^4[n]$ for $n\geq 16$. 
For $n \le 16$, we find out most of the elements of $I^4[n]$.
\end{abstract}
%\textbf{Keywords.} Latin square, intersection of latin squares, $4$-way intersection.
%
%%%%%%%%%%%%%%%%%%%%%%%%%%%%%%%%%%%%%%%%%%%%%%%%%%%%%%
\section{Introduction and Preliminaries}

A {\sf partial latin rectangle}  is an $r \times n$ ($r\leq n$) array such that each cell 
is either empty or consists of a symbol from a set of $n$ distinct symbols (e.g. $\{1,2,\ldots,n\}$),
and that each symbol appears at most once in each row and in each column.
A {\sf latin rectangle} is a partial latin rectangle when all cells are non-empty.
A {\sf (partial) latin square} of order $n$ is an $n \times n$ (partial) latin rectangle.
We assume the set of symbols $\{1,2,\ldots,n\}$ are used in latin squares of order $n$.

A {\sf $\mu$-way latin square} ($\mu \geq 2$) of order $n$ is an
ordered set of $\mu$ latin squares of order $n$ with the following property: 
the $\mu$ entries in cells with the same coordinate are either all the same, or all different.
The cells with the same entries are often called {\sf fixed cells} or {\sf identical cells} interchangeably.
A $\mu$-way latin square has {\sf $k$ intersection} when the number of fixed cells is exactly $k$.
Similarly, $\mu$-way latin rectangle of order $r\times n$ is defined.

For each $n \geq 1$, by  $I^\mu[n]$ we denote the set of all integers $k$ such that there exist
$\mu$-way latin squares of order $n$ with $k$ intersection.
Determining  $I^\mu[n]$ is called {\sf $\mu$-way latin intersection problem}.

The intersection problem has arisen in many combinatorial areas such as latin squares
\cite{MR1874724,MR1125351,Fu},
Steiner triple systems \cite{MR941781}, $m$-cycle systems
\cite{MR1821945}, and design theory
\cite{MR1278954}. For an old survey on intersection problem see
\cite{MR1278954}. The intersection problem for latin squares was
introduced at first by Fu \cite{Fu} for two latin squares.
Later, Fu and Fu \cite{MR1093150} proposed an intersection problem for $3$ latin
squares which is a relaxation of the problem we have defined
above for $\mu=3$. Instead of having entries of
$n^2-k$ cells mutually distinct in all three latin squares, they only
require not all of these entries to be equal. Adams et al. \cite{MR1874724}
have completely determined $I^3[n]$ for $n\geq 1$. There was an error in
\cite{MR1874724}, in showing that $35 \in I^3[8]$. We have found it by a computer program
and shown in Figure \ref{fig:$3$-way latin square-example}.

For $\mu \geq 2$, a {\sf $\mu$-way latin trade of volume $s$} is defined as follows: A group
of $\mu$ partial latin rectangles such that each of them have precisely the same $s$ filled cells.
If cell $(i,j)$ is filled, then its entry is different in all $\mu$ partial latin rectangles. Moreover, for any
relevant $i$, the set of entries of $i^{th}$ row is the same for all $\mu$ partial latin rectangles,
and similarly for relevant columns $j$.  In \cite{MR1904722}, there are some useful results on $\mu$-way
latin trades which we have used for our work.

Note that from each $\mu$-way latin square, by considering all cells which have identical fixed elements
as empty cells, we obtain a $\mu$-way latin trade. We will refer to the set of cells which have fixed elements
as {\sf intersection part} and to its complement as {\sf trade part}.

Given a $\mu$-way latin square (latin trade) ${\cal L}$, its {\sf skeleton} is a binary matrix $S$
where $S_{i,j}$ is $1$ if and only if cell $(i,j)$ is in the intersection part (is an empty cell, respectively).
Denote by $r_k$ and $c_k$ the number of ones in the $k$'th row and $k$'th column of $S$, respectively.
We call $(r_1, r_2, \cdots, r_n)$ and $(c_1, c_2, \cdots, c_n)$ the {\sf row sequence} 
and {\sf column sequence} of ${\cal L}$, respectively.
An example of a $3$-way latin square and its skeleton is given in
Figure \ref{fig:$3$-way latin square-example}.

\begin{figure}[ht]
\begin{tabular}{ccc}

%\hspace{1 cm}
\def\arraystretch{1.4}
\begin{tabular}
{|@{\hspace{7pt}}c@{\hspace{8pt}} |@{\hspace{1pt}}c@{\hspace{1pt}}
|@{\hspace{1pt}}c@{\hspace{1pt}} |@{\hspace{1pt}}c@{\hspace{1pt}}
|@{\hspace{1pt}}c@{\hspace{1pt}} |@{\hspace{1pt}}c@{\hspace{1pt}}
|@{\hspace{1pt}}c@{\hspace{1pt}} |@{\hspace{1pt}}c@{\hspace{1pt}}
|}
\hline
$1$  & $2$  & $3$  & $4$  & $5$  & $6$  & $7$  & $8$  \\
\hline
$2$  & $1$  & $4$  & $3$  & $6$  & $5$  & $8$  & $7$  \\
\hline
$3$  & $4$  & $1$  & $2$  & $7$  & $8$  & $5$  & $6$  \\
\hline
$5$  & $6$  & $7$  & $8$  & \mthree123  & \mthree214  & \mthree341  & \mthree432  \\
\hline
$7$  & $8$  & \mthree265  & \mthree156  & \mthree314  & \mthree421  & \mthree632  & \mthree543  \\
\hline
$6$  & \mthree375  & $8$  & \mthree517  & \mthree241  & \mthree732  & \mthree423  & \mthree154  \\
\hline
$4$  & \mthree537  & \mthree652  & \mthree765  & $8$  & \mthree173  & \mthree216  & \mthree321  \\
\hline
$8$  & \mthree753  & \mthree526  & \mthree671  & \mthree432  & \mthree347  & \mthree164  & \mthree215  \\
\hline
\end{tabular}

&

&

\def\arraystretch{1.4}
\begin{tabular}
{|@{\hspace{7pt}}c@{\hspace{7pt}} |@{\hspace{7pt}}c@{\hspace{7pt}}
|@{\hspace{7pt}}c@{\hspace{7pt}} |@{\hspace{7pt}}c@{\hspace{7pt}}
|@{\hspace{7pt}}c@{\hspace{7pt}} |@{\hspace{7pt}}c@{\hspace{7pt}}
|@{\hspace{7pt}}c@{\hspace{7pt}} |@{\hspace{7pt}}c@{\hspace{7pt}}
|}
\hline
 $1$ & $1$ & $1$ & $1$ & $1$ & $1$ & $1$ & $1$\\ \hline
 $1$ & $1$ & $1$ & $1$ & $1$ & $1$ & $1$ & $1$\\ \hline
 $1$ & $1$ & $1$ & $1$ & $1$ & $1$ & $1$ & $1$\\ \hline
 $1$ & $1$ & $1$ & $1$ & $0$ & $0$ & $0$ & $0$\\ \hline
 $1$ & $1$ & $0$ & $0$ & $0$ & $0$ & $0$ & $0$\\ \hline
 $1$ & $0$ & $1$ & $0$ & $0$ & $0$ & $0$ & $0$\\ \hline
 $1$ & $0$ & $0$ & $0$ & $1$ & $0$ & $0$ & $0$\\ \hline
 $1$ & $0$ & $0$ & $0$ & $0$ & $0$ & $0$ & $0$\\ \hline
\end{tabular}

\end{tabular}
\caption{A $3$-way latin square of order $8$ with $35$ fixed cells and its skeleton}
\label{fig:$3$-way latin square-example}
\end{figure}

Next, we present some old (which are referenced) and new results which are used to determine $I^4[n]$ for $n\geq 1$.

\begin{lemma} {\em \cite{MR1904722}}
\label{lemma2.2.2 Hadi Thesis}
If $\min\{m,n\}\geq \mu$ then there exists a $\mu$-way latin trade of order $m \times n$
of volume $mn$.
\end{lemma}
Let ${\cal T}$ be a $\mu$-way latin trade of order $n$.
In the following results, $R_i$ and $C_j$ denote the set of
elements of row $i$ and column $j$ of ${\cal T}$, respectively.

The next lemma is an immediate result from the definition of $\mu$-way latin trades.

\begin{lemma} {\em \cite{MR1904722}}{
\label{lemma2.2.3 Hadi Thesis}
Let ${\cal T}$ have a nonempty cell $(i,j)$.
Then $|R_i\cap C_j|\geq \mu.$
}\end{lemma}
\begin{corollary} {\em \cite{MR1904722}}{
\label{corollary2.2.2 Hadi Thesis}
Let ${\cal T}$ have a nonempty cell $(i,j)$.
If $|R_i|=|C_j|=\mu$ then $R_i=C_j$.
}\end{corollary}
\begin{corollary}{
\label{corollary2.2.3 Hadi Thesis}
Assume $i_1 \not= i_2$ and cells $(i_1, j), (i_2, j)$ are nonempty.
If $|C_j|=\mu+1$ and $|R_{i_1}|=|R_{i_2}|=\mu$ then
$$R_{i_1}\cup R_{i_2}=C_j, \quad |R_{i_1}\cap R_{i_2}|=\mu-1.$$

Similarly, assume that $j_1 \not= j_2$ and cells $(i, j_1), (i, j_2)$ are nonempty.
If $|R_i|=\mu+1$ and $|C_{j_1}|=|C_{j_2}|=\mu$, then
$$C_{j_1}\cup C_{j_2}=R_i, \quad |C_{j_1}\cap C_{j_2}|=\mu-1.$$
}\end{corollary}
\begin{proof}{
By symmetry, it suffices to show only the first part. According to
Lemma~\ref{lemma2.2.3 Hadi Thesis}
we have $R_{i_1},R_{i_2}\subseteq C_j$. Therefore, there are two possibilities:
$|R_{i_1}\cap R_{i_2}|=\mu-1$, or $|R_{i_1}\cap R_{i_2}|=\mu$ (in which case $R_{i_1} = R_{i_2}$).
We show the latter case is not possible. If the latter case happens then
there exists an element $x \in C_j\setminus (R_{i_1}\cup R_{i_2})$. This means that
$x$ can appear in at most $\mu-1$ cells of column $j$, which is a
contradiction as each element should appear in exactly $\mu$
cells of a row or column.
}\end{proof}

\begin{lemma}
\label{lemma:possible-in-seq}
If an element appears at least $n-\mu+1$ times in the intersection part of a $\mu$-way latin square ${\cal L}$,
then it appears only in the intersection part of ${\cal L}$.
\end{lemma}
\begin{proof}{
By Lemma~\ref{lemma2.2.3 Hadi Thesis}, if an element appears in the trade,
then it appears in at least $\mu$ rows (and $\mu$ columns) of each partial latin square in the trade.
Hence, it occurs at most $n-\mu$ times in the intersection part.
%Hence, there are at most $n-\mu$ available positions for that element in the intersection part.
}\end{proof}
%%%%%%%%%%%%%%%%%%%%%%%%%%%%%%%%%%%%%%%%%%%%%%%%
In \cite{MR1874724} it is shown that 
$I^3[2n]\supseteq I^3[n]+I^3[n]+I^3[n]+I^3[n]$ for $n \geq 1$.
This can be simply generalized to $I^{\mu}[2n]$ for any $\mu \geq 4$. Using this generalization,
together with the fact that any latin square of order $i$ can be
embedded in a latin square of order $2n$ when $i\leq n$, we
obtain the following proposition.
\begin{proposition}{
\label{theorem2.3.1 Hadi Thesis}
For any $n\geq 1$ and $\mu \geq 2$ we have
$$I^{\mu}[2n]\supseteq \Big(I^{\mu}[n]+I^{\mu}[n]+I^{\mu}[n]+I^{\mu}[n]\Big)\cup
\Big(\bigcup_{i=1}^{n}{\big(I^{\mu}[i]+\{(2n)^2-i^2\}\big)}\Big).$$
}\end{proposition}

\begin{lemma}{
\label{lemma: transversals}
Let ${\cal L}$ be a $\mu$-way latin square~of order $n$ with $k$ fixed cells.
If $p$ is the number of elements of ${\cal L}$ which appear only in the intersection part then
$$ p \geq \left\lceil n - {{n^2-k}\over{\mu}} \right\rceil .$$
}\end{lemma}

\begin{proof}{
From the definition of a $\mu$-way latin square, each of the other $n-p$ elements
appear at least $\mu$ times in the cells of the trade part of ${\cal L}$.
So, each of these $n-p$ elements can appear at most $n-\mu$ times in the
cells of the intersection part.
Hence, $k\leq pn + (n-p)(n-\mu)$, or equivalently, ${\mu}p \geq {\mu}n -n^2 + k$.  
}\end{proof}

%----------------------------------------------------------------------------------------------

\begin{lemma}{
\label{lemma: sequence 4-7}
Let ${\cal L}$ be a $4$-way latin square of order $7$ and $a=(a_1,a_2,\cdots,a_7)$ be
its row (or column) sequence. None of the sequences $\{7,3\}$, $\{7,7,2\}$ and $\{7,7,7,1\}$
can be a subsequence of $a$.
}\end{lemma}
\begin{proof}{
We show that $\{7,3\}$ can not be a subsequence of the row sequence of any $4$-way latin square of order $7$.
Other statements are similar.
Suppose, on the contrary, that there is such an ${\cal L}$.
By a permutation on row and columns of ${\cal L}$, we may assume that the first
and second rows have $7$ and $3$ fixed cells, respectively, and the fixed cells of
the second row are located at the first $3$ columns. By a permutation on the elements,
we may assume that the first row is filled as $1,2,\ldots,7$ in order,
and the $3$ fixed elements of the second row are $x$, $y$, and $z$ (see Figure \ref{fig: 7-3 fix}).\\
\begin{figure}[ht]
\vskip 0.5cm
\begin{center}
\begin{tabular}{|c|c|c|c|c|c|c|} \hline
1 & 2 & 3 & 4 & 5 & 6 & 7 \\
\hline
$x$ & $y$ & $z$ & & & & \\
\hline
\end{tabular}
\end{center}
\caption{A possible $\{7,3\}$ sequence.}
\label{fig: 7-3 fix}
\end{figure}
\\
Consider the cell $(2,4)$ in the trade part. Since there should appear
four different elements distinct from $x,y,z$ and $4$, we should have
$4\in \{x,y,z\}$. By a similar argument for the rest of cells in
the second row, we have $5,6,7\in \{x,y,z\}$ which is impossible.}\end{proof}
With the same approach as in the above, we can show the following two
lemmata as well.
\begin{lemma}{
\label{lemma: sequence 4-6}
Let ${\cal L}$ be a $4$-way latin square of order $6$ and $a=(a_1,a_2,\cdots,a_6)$ be
its row (or column) sequence. None of the sequences $\{6,2\}$ and $\{6,6,1\}$
can be a subsequence of $a$.
}\end{lemma}
\begin{lemma}{
\label{lemma: sequence 4-5}
Let ${\cal L}$ be a $4$-way latin square of order $5$ and $a=(a_1,a_2,\cdots,a_5)$ be
its row (or column) sequence. The sequence $\{5,1\}$ cannot be a subsequence of $a$.
}\end{lemma}
%

%%%%%%%%%%%%%%%%%%%%%%%%%%%%%%%%%%%%%%%%%%%%%%%%%%%%%%%%%%%%%%%%%%%%%%%%%
\section{Constructions \label{construction}}

%\setcounter{theorem}{0}
%\setcounter{preproposition}{0}
%\setcounter{pretechnique}{0}
%\setcounter{lemma}{0}
%\setcounter{precorollary}{0}
%\setcounter{prelemma}{0}
%\setcounter{preexample}{0}
%
%%%%%%%%%%%%%%%%%%%%%%%%%%%%%%%%%%%%%%%%%%%%%%%%%%%%%%%%%%%%%%%%%%%%%%%%%%%
In this section, we introduce four techniques which contribute to the generation of the
majority of $4$-way intersections. The first technique is inspired by \cite{MR1874724} and the rest are new.
We start with illustrating the first technique by an example, then elaborating the technique in the sequel.
\begin{example}{
Consider the following partial $4$-way latin squares ${\cal A}$,
${\cal B}$, and~${\cal C}$. 

\label{example2.3.2 Hadi Thesis}
%
%Figure of example 2.3.2 from Hadi's Thesis.
%\begin{figure}[ht]

\begin{footnotesize}
\vspace{.5cm}
\def\arraystretch{1.7}
$$
{\cal A}=
%\hspace{-.75cm}
\begin{tabular}
{|@{\hspace{.1pt}}c@{\hspace{.1pt}}
|@{\hspace{.1pt}}c@{\hspace{.1pt}}
|@{\hspace{.1pt}}c@{\hspace{.1pt}}
|@{\hspace{.1pt}}c@{\hspace{.1pt}}
||@{\hspace{.1pt}}c@{\hspace{.1pt}}
|@{\hspace{.1pt}}c@{\hspace{.1pt}}
|@{\hspace{.1pt}}c@{\hspace{.1pt}}
|@{\hspace{.1pt}}c@{\hspace{.1pt}}
|@{\hspace{.1pt}}c@{\hspace{.1pt}}|} \hline
\x. &\x. &\x. &\x. & \mfour2345 & \mfour1453 & \mfour4132 & \mfour5214 & \mfour3521 \\ \hline
\x. &\x. &\x. &\x. & \mfour3254 & \mfour4512 & \mfour5321 & \mfour1435 & \mfour2143 \\ \hline
\x. &\x. &\x. &\x. & \mfour4523 & \mfour5231 & \mfour2415 & \mfour3142 & \mfour1354 \\ \hline
\x. &\x. &\x. &\x. & \mfour5432 & \mfour3124 & \mfour1543 & \mfour2351 & \mfour4215 \\ \hline \hline
\mfour2345 & \mfour3254 & \mfour4523 & \mfour5432 &$1$ &\x. &\x. &\x. &\x.\\ \hline
\mfour1453 & \mfour4512 & \mfour5231 & \mfour3124 &\x. &\mfour2345 &\x. &\x. &\x.\\ \hline
\mfour4132 & \mfour5321 & \mfour2415 & \mfour1543 &\x. &\x. &\mfour3254 &\x. &\x.\\ \hline
\mfour5214 & \mfour1435 & \mfour3142 & \mfour2351 &\x. &\x. &\x. &\mfour4523 &\x.\\ \hline
\mfour3521 & \mfour2143 & \mfour1354 & \mfour4215 &\x. &\x. &\x. &\x. &\mfour5432 \\ \hline
\end{tabular}
%\hspace{.6cm}
\ \
%%%%%%%%%%%%%%%%%%%%
\ \ {\cal B}=
\def\arraystretch{1.46}
\begin{tabular}
{|@{\hspace{6.4pt}}c@{\hspace{6.4pt}}
|@{\hspace{6.4pt}}c@{\hspace{6.4pt}}
|@{\hspace{6.4pt}}c@{\hspace{6.4pt}}
|@{\hspace{6.4pt}}c@{\hspace{6.4pt}}
||@{\hspace{7.4pt}}c@{\hspace{7.4pt}}
|@{\hspace{7.4pt}}c@{\hspace{7.4pt}}
|@{\hspace{7.4pt}}c@{\hspace{7.4pt}}
|@{\hspace{7.4pt}}c@{\hspace{7.4pt}}
|@{\hspace{7.4pt}}c@{\hspace{7.4pt}}|} \hline
$9$ &$6$ &$7$ &$8$ &\x. &\x. &\x. &\x. &\x.\\ \hline
$8$ &$9$ &$6$ &$7$ &\x. &\x. &\x. &\x. &\x.\\ \hline
$7$ &$8$ &$9$ &$6$ &\x. &\x. &\x. &\x. &\x.\\ \hline
$6$ &$7$ &$8$ &$9$ &\x. &\x. &\x. &\x. &\x.\\ \hline \hline
\x. &\x. &\x. &\x. &\x. &\x. &\x. &\x. &\x.\\ \hline
\x. &\x. &\x. &\x. &\x. &\x. &\x. &\x. &\x.\\ \hline
\x. &\x. &\x. &\x. &\x. &\x. &\x. &\x. &\x.\\ \hline
\x. &\x. &\x. &\x. &\x. &\x. &\x. &\x. &\x.\\ \hline
\x. &\x. &\x. &\x. &\x. &\x. &\x. &\x. &\x.\\ \hline
\end{tabular}
$$
\end{footnotesize}

and

$$
{\cal C}=
\def\arraystretch{1.55}
\begin{tabular}
{|@{\hspace{9.5pt}}c@{\hspace{9.5pt}}
|@{\hspace{9.5pt}}c@{\hspace{9.5pt}}
|@{\hspace{9.5pt}}c@{\hspace{9.5pt}}
|@{\hspace{9.5pt}}c@{\hspace{9.5pt}}
||@{\hspace{.1pt}}c@{\hspace{.1pt}}
|@{\hspace{.1pt}}c@{\hspace{.1pt}}
|@{\hspace{.1pt}}c@{\hspace{.1pt}}
|@{\hspace{.1pt}}c@{\hspace{.1pt}}
|@{\hspace{.1pt}}c@{\hspace{.1pt}}|} \hline
\x. &\x. &\x. &\x. &\x. &\x. &\x. &\x. &\x. \\ \hline
\x. &\x. &\x. &\x. &\x. &\x. &\x. &\x. &\x. \\ \hline
\x. &\x. &\x. &\x. &\x. &\x. &\x. &\x. &\x. \\ \hline
\x. &\x. &\x. &\x. &\x. &\x. &\x. &\x. &\x. \\ \hline \hline
\x. &\x. &\x. &\x. &\x. &\mfour6789 &\mfour7896 &\mfour8967 &\mfour9678 \\ \hline
\x. &\x. &\x. &\x. &\mfour9678 &\x. &\mfour6789 &\mfour7896 &\mfour8967 \\ \hline
\x. &\x. &\x. &\x. &\mfour8967 &\mfour9678 &\x. &\mfour6789 &\mfour7896 \\ \hline
\x. &\x. &\x. &\x. &\mfour7896 &\mfour8967 &\mfour9678 &\x. &\mfour6789 \\ \hline
\x. &\x. &\x. &\x. &\mfour6789 &\mfour7896 &\mfour8967 &\mfour9678 &\x. \\ \hline
\end{tabular}
$$
}\end{example}

By combining these partial $4$-way latin squares,
we obtain a $4$-way latin square of order $9$ with $17$ fixed cells.

\begin{technique}{{\bf [$n \rightarrow 2n+1$ technique]}
\label{tech1}
%technique 1 in Hadi's Thesis.
%This technique is almost a combination of techniques stated in
%\cite{MR1874724}. 
This technique constructs a $\mu$-way latin square~of order $2n+1$ by 
generating and combining three partial $\mu$-way latin squares ${\cal A}$, ${\cal B}$, and ${\cal C}$.
Partial latin squares ${\cal A}$, ${\cal B}$, and ${\cal C}$ are generated as follows.
Let ${\cal A'}$ be a $\mu$-way latin square~of order $n+1$ with elements from $\{1,\ldots,n+1\}$.
We construct ${\cal A}$ by embedding,
symmetrically, the first $n$ rows of ${\cal A'}$, at the top-right and bottom-left
corners of a square of order $2n+1$ and laying the $(n+1)^{th}$ row of ${\cal A'}$ at
the down-right corner, diagonally.\\
${\cal B}$ is constructed by embedding a $\mu$-way latin square~${\cal B'}$ of order $n$
with elements from $\{n+2,\ldots,2n+1\}$ at the top-left corner of a square of order $2n+1$.\\
${\cal C}$ is made by embedding a partial $\mu$-way latin square~of order $n+1$, say ${\cal C'}$,
at the down-right corner of a square of order $2n+1$.
Note that the elements of ${\cal C'}$ are from $\{n+2,\ldots,2n+1\}$
and diagonal cells of ${\cal C'}$ are empty.
}\end{technique}
The following lemma is a generalization of Lemma 2.3 in \cite{MR1874724}. As in
Proposition \ref{theorem2.3.1 Hadi Thesis}, by the fact that any latin square of order $i$ can be
embedded in a latin square of order $2n+1$ when $i\leq n$ we have
$\bigcup_{i=1}^{n}{(I^4[i]+\{(2n+1)^2-i^2\})} \subseteq I^4[2n+1]$.
\begin{lemma}
\label{theorem2.3.2 Hadi Thesis}
If $n\geq 4$ then\\
$$I^4[2n+1] \supseteq \{I^4[n]+(n+1)\{[0,n-4]\cup \{n\}\}+C\} \cup X.$$
where
$C=\bigcup_{t=1}^{n-3}\{2tn,2tn-t,2tn-n\} \cup \{0,1,2\} \cup  (2n+1)\{[0,n-3]\cup \{n+1\}\} \cup (n+1)\{[1,2n-7]\cup [n+1,2n-3]\}$ and
$X = \bigcup_{i=1}^{n}{(I^4[i]+\{(2n+1)^2-i^2\})}$.
%$X = \cup(\bigcup_{i=1}^{n}{(I^4[i]+\{(2n+1)^2-i^2\})})$.
\end{lemma}
%
%%%%%%%%%%%%%%%%%%%%%%%%%%%%%%%%%
\begin{technique}{{\bf [Trade-into-Trade technique]}
\label{tech2}
%technique 2 in Hadi's thesis
\noindent In this technique we consider a $\mu'$-way latin square of order $n$.
Then for each $i$,
$1\leq i\leq \mu'$, we substitute each entry of unfixed cells in the $i^{th}$ latin square
with a proper $\mu_i$-way latin trade of order $m$. In this way we obtain a
$(\sum_{i=1}^{\mu'}\mu_i)$-way latin square of order $mn$.

\noindent Let's illustrate a simple case of this method with the following example.
\begin{example}{
\label{example2.3.4 Hadi Thesis}
%\hspace{2cm} example 2.3.4 in Hadi's Thesis.
Consider the following $2$-way latin square of order $4$ with $9$ fixed cells.

\def\arraystretch{1.8}
\begin{center}
\begin{tabular}
{|@{\hspace{7pt}}c@{\hspace{7pt}}
|@{\hspace{1pt}}c@{\hspace{1pt}}
|@{\hspace{1pt}}c@{\hspace{1pt}}
|@{\hspace{1pt}}c@{\hspace{1pt}}|} \hline
${A_1}$ &${A_2}$ &${A_3}$ &${A_4}$ \\ \hline
${A_3}$ &${A_4}$ &${A_1}_{A_2}$ &${A_2}_{A_1}$ \\ \hline
${A_2}$ &${A_3}_{A_1}$ &${A_4}$ &${A_1}_{A_3}$ \\ \hline
${A_4}$ &${A_1}_{A_3}$ &${A_2}_{A_1}$ &${A_3}_{A_2}$ \\ \hline
\end{tabular}
\end{center}
}\end{example}
We replace each $A_i$, $1\leq i\leq 4$, with a $2$-way latin trade. Note that
elements of any two trades corresponding to two different $A_i$ are disjoint.
This way, we obtain the following $4$-way latin square of order $8$ with $36$ fixed cells.\\

\def\arraystretch{1.3}
\begin{center}
\begin{tabular}
{|@{\hspace{.17pt}}c@{\hspace{.17pt}} |@{\hspace{.17pt}}c@{\hspace{.17pt}}
|@{\hspace{.17pt}}c@{\hspace{.17pt}} |@{\hspace{.17pt}}c@{\hspace{.17pt}}
|@{\hspace{.17pt}}c@{\hspace{.17pt}} |@{\hspace{.17pt}}c@{\hspace{.17pt}}
|@{\hspace{.17pt}}c@{\hspace{.17pt}} |@{\hspace{.17pt}}c@{\hspace{.17pt}}
|}
\hline
%\multicolumn{8}{|@{\hspace{2pt}}l@{\hspace{2pt}}|@{\hspace{0pt}}}{$8, 36$}\\
%\hline \hline
\n1  & \n2  & \n3  & \n4  & \n5  & \n6  &
\n7  & \n8  \\
\hline
\n2  & \n1  & \n4  & \n3  & \n6  & \n5  &
\n8  & \n7  \\
\hline \hline
\n5  & \n6  & \n7  & \n8  & \r12  & \r21  &
\r34  & \r43  \\
\hline
\n6  & \n5  & \n8  & \n7  & \r21  & \r12  &
\r43  & \r34  \\
\hline \hline
\n3  & \n4  & \r56  & \r65  & \n7  & \n8  &
\r12  & \r21  \\
\hline
\n4  & \n3  & \r65  & \r56  & \n8  & \n7  &
\r21  & \r12  \\
\hline \hline
\n7  & \n8  & \r12  & \r21  & \r34  & \r43  &
\r56  & \r65  \\
\hline
\n8  & \n7  & \r21  & \r12  & \r43  & \r34  &
\r65  & \r56  \\
\hline % \hline
\end{tabular}
\hspace{.5cm}
\begin{tabular}
{|@{\hspace{.17pt}}c@{\hspace{.17pt}} |@{\hspace{.17pt}}c@{\hspace{.17pt}}
|@{\hspace{.17pt}}c@{\hspace{.17pt}} |@{\hspace{.17pt}}c@{\hspace{.17pt}}
|@{\hspace{.17pt}}c@{\hspace{.17pt}} |@{\hspace{.17pt}}c@{\hspace{.17pt}}
|@{\hspace{.17pt}}c@{\hspace{.17pt}} |@{\hspace{.17pt}}c@{\hspace{.17pt}}
|}
\hline
%\multicolumn{8}{|@{\hspace{2pt}}l@{\hspace{2pt}}|@{\hspace{0pt}}}{$8, 36$}\\
%\hline \hline
\n1  & \n2  & \n3  & \n4  & \n5  & \n6  &
\n7  & \n8  \\
\hline
\n2  & \n1  & \n4  & \n3  & \n6  & \n5  &
\n8  & \n7  \\
\hline \hline
\n5  & \n6  & \n7  & \n8  & \r34  & \r43 &
\r12  & \r21  \\
\hline
\n6  & \n5  & \n8  & \n7  & \r43  & \r34 &
\r21  & \r12  \\
\hline \hline
\n3  & \n4  & \r12  & \r21 & \n7  & \n8  &
\r56  & \r65  \\
\hline
\n4  & \n3  & \r21  & \r12  & \n8  & \n7 &
\r65  & \r56  \\
\hline \hline
\n7  & \n8  & \r56  & \r65  & \r12  & \r21  &
\r34  & \r43  \\
\hline
\n8  & \n7  & \r65  & \r56  & \r21  & \r12  &
\r43  & \r34  \\
\hline %\hline
\end{tabular}
%\begin{figure}[ht]
%\label{infrastructure trade} \vspace*{0mm} \caption{another caption}
%\end{figure}
\end{center}

More generally, this technique can be used to construct {\sf fine
structures} which are defined in \cite{MR2202131}.
}\end{technique}

%%%%%%%%%%%%%%%%%%%%%%%%%%%%%%%%%
\noindent Let's consider an example before explaining next technique.
\begin{example}{
\label{example2.3.5 Hadi Thesis}
%\hspace{2cm} Example 2.3.5 from Hadi's Thesis.
Since the following three partial latin squares are completable to
a latin square of order $9$ we have $46\in I^3[9]$.
\def\arraystretch{.9}
%\begin{center}
$$
{\cal A}_1=
\begin{tabular}
{|@{\hspace{4pt}}c@{\hspace{4pt}}
|@{\hspace{4pt}}c@{\hspace{4pt}}
|@{\hspace{4pt}}c@{\hspace{4pt}}
|@{\hspace{4pt}}c@{\hspace{4pt}}
|@{\hspace{4pt}}c@{\hspace{4pt}}
||@{\hspace{4pt}}c@{\hspace{4pt}}
|@{\hspace{4pt}}c@{\hspace{4pt}}
|@{\hspace{4pt}}c@{\hspace{4pt}}
|@{\hspace{4pt}}c@{\hspace{4pt}}|}
\ &\ &\ &\ &\ &\ &\ &\ &\ \\[-1.5 mm] \hline 
\ &\ &\ &\ &\ &\ &\ &\ &\ \\ \hline
\ &\ &\ &\ &\ &$\bf 1$ &$\bf 2$ &$\bf 3$ &$\bf 4$ \\ \hline
\ &\ &\ &\ &\ &$\bf 4$ &$\bf 1$ &$\bf 2$ &$\bf 3$\\ \hline
$5$ &$6$ &$7$ &$8$ &$9$ &$\bf 3$ &$\bf 4$ &$\bf 1$ &$\bf 2$ \\ \hline
$9$ &$5$ &$6$ &$7$ &$8$ &$\bf 2$ &$\bf 3$ &$\bf 4$ &$\bf 1$ \\ \hline
$1$ &$2$ &$3$ &$4$ &$5$ &$\bf 6$ &$\bf 7$ &$\bf 8$ &$\bf 9$ \\ \hline
\end{tabular}
\hspace{.5cm}
\ {\cal A}_2=
\begin{tabular}
{|@{\hspace{4pt}}c@{\hspace{4pt}}
|@{\hspace{4pt}}c@{\hspace{4pt}}
|@{\hspace{4pt}}c@{\hspace{4pt}}
|@{\hspace{4pt}}c@{\hspace{4pt}}
|@{\hspace{4pt}}c@{\hspace{4pt}}
||@{\hspace{4pt}}c@{\hspace{4pt}}
|@{\hspace{4pt}}c@{\hspace{4pt}}
|@{\hspace{4pt}}c@{\hspace{4pt}}
|@{\hspace{4pt}}c@{\hspace{4pt}}|}
\ &\ &\ &\ &\ &\ &\ &\ &\ \\[-1.5 mm] \hline
\ &\ &\ &\ &\ &\ &\ &\ &\ \\ \hline
\ &\ &\ &\ &\ &$\bf 4$ &$\bf 1$ &$\bf 2$ &$\bf 3$ \\ \hline
\ &\ &\ &\ &\ &$\bf 3$ &$\bf 4$ &$\bf 1$ &$\bf 2$ \\ \hline
$9$ &$5$ &$6$ &$7$ &$8$  &$\bf 2$ &$\bf 3$ &$\bf 4$ &$\bf 1$ \\ \hline
$1$ &$2$ &$3$ &$4$ &$5$  &$\bf 6$ &$\bf 7$ &$\bf 8$ &$\bf 9$ \\ \hline
$5$ &$6$ &$7$ &$8$ &$9$  &$\bf 1$ &$\bf 2$ &$\bf 3$ &$\bf 4$ \\ \hline
\end{tabular}
%\end{center}
\vspace{.5cm}
$$
$$
{\cal A}_3=
\begin{tabular}
{|@{\hspace{4pt}}c@{\hspace{4pt}}
|@{\hspace{4pt}}c@{\hspace{4pt}}
|@{\hspace{4pt}}c@{\hspace{4pt}}
|@{\hspace{4pt}}c@{\hspace{4pt}}
|@{\hspace{4pt}}c@{\hspace{4pt}}
||@{\hspace{4pt}}c@{\hspace{4pt}}
|@{\hspace{4pt}}c@{\hspace{4pt}}
|@{\hspace{4pt}}c@{\hspace{4pt}}
|@{\hspace{4pt}}c@{\hspace{4pt}}|}
\ &\ &\ &\ &\ &\ &\ &\ &\ \\[-1.5mm] \hline
\ &\ &\ &\ &\ &\ &\ &\ &\ \\ \hline
\ &\ &\ &\ &\ &$\bf 3$ &$\bf 4$ &$\bf 1$ &$\bf 2$ \\ \hline
\ &\ &\ &\ &\ &$\bf 2$ &$\bf 3$ &$\bf 4$ &$\bf 1$ \\ \hline
$1$ &$2$ &$3$ &$4$ &$5$ &$\bf 6$ &$\bf 7$ &$\bf 8$ &$\bf 9$ \\ \hline
$5$ &$6$ &$7$ &$8$ &$9$ &$\bf 1$ &$\bf 2$ &$\bf 3$ &$\bf 4$ \\ \hline
$9$ &$5$ &$6$ &$7$ &$8$ &$\bf 4$ &$\bf 1$ &$\bf 2$ &$\bf 3$ \\ \hline
\end{tabular}
$$
}\end{example}
\begin{technique}{{\bf [Gear technique 1]}
\label{tech3}
%technique 3 in Hadi's thesis
Consider a completable partial latin square of order $n$, as in the following figure,
where $A=\{1,\ldots,a\}$, $B=\{a+1,\ldots,n\}$ and $A\cup B$ are the
sets of entries at the
specified portion. Then by cyclically permuting the rows of the two
subrectangles of order $y\times b$ and $(y+x)\times a$, we obtain
a $\mu$-way latin square of order $n$ (after filling the empty cells identically, in all latin
squares).\\
%------Figure 2.8 from Hadi's Thesis------\\
\def\sp{\hspace{1.2cm}}
\def\spp{\hspace{2.4cm}}
\def\d{\displaystyle}

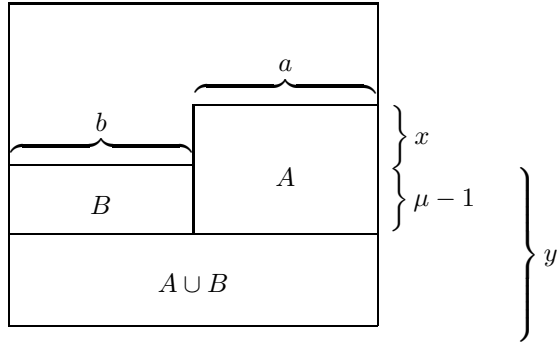
\def\arraystretch{1}
$$
\begin{array}{|c|c|cc}
\cline{1-2}
\multicolumn{2}{|c|}{}\\[3mm]
\multicolumn{1}{|c}{}&\multicolumn{1}{r|}{\hspace{-1.5mm}\overbrace{\spp}^{{\d a}}\hspace{-1.5mm}}\\
\cline{2-2}
\multicolumn{1}{|c|}{}&\multicolumn{1}{|c|}{}&
\hspace{-4mm}\left.\begin{array}{c}
\\[4mm]
\end{array}\right\}x\\[-7.8mm]
\multicolumn{1}{|c}{\hspace{-1.5mm}\overbrace{\spp}^{{\d b}}\hspace{-1.5mm}}&\multicolumn{1}{|c|}{}\\
\cline{1-1}
&&&
\hspace{-23.8mm}\left.\begin{array}{c}
\\[5mm]
\end{array}\right\}\mu-1\\
\cline{1-2}
\multicolumn{2}{|c|}{}\\[8mm]
\cline{1-2}
\multicolumn{3}{c}{}\\[-6.5mm]  % A, B, A\cupB, y
\multicolumn{1}{c}{}\\[-23mm]
\multicolumn{3}{c}{}&
\hspace{4mm}\left.\begin{array}{c}
\\[1.9cm]
\end{array}\right\}y\\[-2.0cm]
\multicolumn{1}{c}{B}\\[-0.8cm] % A, B, A\cupB
\multicolumn{1}{r}{}&\multicolumn{1}{c}{A}\\[1.cm]
\multicolumn{2}{c}{A\cup B}
\end{array}
$$\\

Sufficient conditions to have such a completable partial latin square are

\begin{enumerate}{
 \item[(1)] $a\geq x+\mu-1$,
 \item[(2)] $b\geq x+\mu-1$,
 \item[(3)] $a+b=n$,
 \item[(4)] $x\geq 1$ and
 \item[(5)] $y\geq \mu$.
}\end{enumerate}
Briefly, conditions (1), (2), and (3) ensure that latin subrectangles using 
elements of $A$, $B$ and $A\cup B$ can 
be constructed, conditions (2) and (3) guarantee that the partial latin square is completable,
condition (4) is intrinsic in the technique, and condition (5) is needed when permuting the rows.
%Note that filling the empty cells results in a $\mu$-way latin square~of order $n$ with
%$n^2-(a(x+y)+by)$ fixed cells.
}\end{technique}

Clearly, we can obtain $\mu$-way latin {\em trade}s using Technique \ref{tech3} and since we don't require the completablity of the partial latin square for latin trades, we can relax the condition $b\geq x+\mu-1$ to $b\geq \mu-1$ and obtain the following Proposition which is used for confining possible members of $I^4[n]$ in the next section.
%relaxes to $b\geq \mu-1$ and $a\geq x+\mu-1$.
%\end{remark}
%
\begin{proposition}{
\label{corollary2.2.4 Hadi Thesis}
For any $i \in \{0,\ldots,\mu\}$, there exists a $\mu$-way latin trade of volume
$s\in \{\mu(3\mu-i),\mu(3\mu-i)+1,\ldots, \mu(3\mu - i) + (\mu - i)\}$.
}\end{proposition}
\begin{proof}{{\bf (Sketch)}
For each $i \in \{0,\ldots,\mu\}$, consider Technique \ref{tech3} for generating latin trades, with the following parameters: \\
\begin{tabular}{cccl}
 & & & $x=1, n=3\mu-i-1, y=\mu \quad and \quad a \in \{\mu,\mu+1,\ldots,2\mu-i\}$
\end{tabular}
}
\end{proof}
As proved in \cite{MR1904722}, there exists a $\mu$-way latin trade of volume $s$
for any $s\geq 3\mu^2+\mu-1$. Hence we get the following
corollary.
\begin{corollary}{
\label{corollary2.3.1 Hadi Thesis}
For any $s\geq 3\mu^2-\mu$, there exists a $\mu$-way latin trade of volume $s$.
}\end{corollary}
%

%%%%%%%%%%%%%%%%%%%%%%%%%%%%%%%%%
\noindent The next technique is similar to Technique \ref{tech3}. As for Technique \ref{tech3},
we start with an example to explain the technique.
\begin{example}{
\label{example2.3.6 Hadi Thesis}
%Example 2.3.6 from Hadi's Thesis
Since the following partial latin squares are completable to
a latin square of order $13$ we have $128\in I^4[13]$.
\def\arraystretch{.8}
$$
{\cal B}_1=
\begin{tabular}
{|@{\hspace{6pt}}c@{\hspace{6pt}}
|@{\hspace{6pt}}c@{\hspace{6pt}}
|@{\hspace{6pt}}c@{\hspace{6pt}}
|@{\hspace{6pt}}c@{\hspace{6pt}}
||@{\hspace{2.3pt}}c@{\hspace{2.3pt}}
|@{\hspace{2.3pt}}c@{\hspace{2.3pt}}
|@{\hspace{2.3pt}}c@{\hspace{2.3pt}}
|@{\hspace{2.3pt}}c@{\hspace{2.3pt}}
||@{\hspace{2.5pt}}c@{\hspace{2.5pt}}
|@{\hspace{2.5pt}}c@{\hspace{2.5pt}}
|@{\hspace{2.5pt}}c@{\hspace{2.5pt}}
|@{\hspace{2.5pt}}c@{\hspace{2.5pt}}
|@{\hspace{2.5pt}}c@{\hspace{2.5pt}}|}
\ &\ &\ &\ &\ &\ &\ &\ &\ &\ &\ &\ &\ \\[-1.5 mm] \hline
\ &\ &\ &\ &\ &\ &\ &\ &\ &\ &\ &\ &\ \\ \hline
\ &\ &\ &\ &\ &\ &\ &\ &$\bf 1$ &$\bf 2$ &$\bf 3$ &$\bf 4$ &$\bf 5$ \\ \hline
\ &\ &\ &\ &$6$ &$7$ &$8$ &$9$ &$\bf 5$ &$\bf 1$ &$\bf 2$ &$\bf 3$ &$\bf 4$ \\
\hline
\ &\ &\ &\ &$9$ &$6$ &$7$ &$8$ &$\bf 4$ &$\bf 5$ &$\bf 1$ &$\bf 2$ &$\bf 3$\\
\hline
\ &\ &\ &\ &$8$ &$9$ &$6$ &$7$ &$\bf 3$ &$\bf 4$ &$\bf 5$ &$\bf 1$ &$\bf 2$ \\
\hline
\ &\ &\ &\ &$7$ &$8$ &$2$ &$3$ &$\bf 9$ &$\bf 6$ &$\bf 4$ &$\bf 5$ &$\bf 1$  \\
 \hline
\end{tabular}
%\hspace{.5cm}
\ {\cal B}_2=
\begin{tabular}
{|@{\hspace{6pt}}c@{\hspace{6pt}}
|@{\hspace{6pt}}c@{\hspace{6pt}}
|@{\hspace{6pt}}c@{\hspace{6pt}}
|@{\hspace{6pt}}c@{\hspace{6pt}}
||@{\hspace{2.3pt}}c@{\hspace{2.3pt}}
|@{\hspace{2.3pt}}c@{\hspace{2.3pt}}
|@{\hspace{2.3pt}}c@{\hspace{2.3pt}}
|@{\hspace{2.3pt}}c@{\hspace{2.3pt}}
||@{\hspace{2.5pt}}c@{\hspace{2.5pt}}
|@{\hspace{2.5pt}}c@{\hspace{2.5pt}}
|@{\hspace{2.5pt}}c@{\hspace{2.5pt}}
|@{\hspace{2.5pt}}c@{\hspace{2.5pt}}
|@{\hspace{2.5pt}}c@{\hspace{2.5pt}}|}
\ &\ &\ &\ &\ &\ &\ &\ &\ &\ &\ &\ &\ \\[-1.5 mm] \hline
\ &\ &\ &\ &\ &\ &\ &\ &\ &\ &\ &\ &\ \\ \hline
\ &\ &\ &\ &\ &\ &\ &\ &$\bf 5$ &$\bf 1$ &$\bf 2$ &$\bf 3$ &$\bf 4$ \\ \hline
\ &\ &\ &\ &$9$ &$6$ &$7$ &$8$ &$\bf 4$ &$\bf 5$ &$\bf 1$ &$\bf 2$ &$\bf 3$ \\
\hline
\ &\ &\ &\ &$8$ &$9$ &$6$ &$7$ &$\bf 3$ &$\bf 4$ &$\bf 5$ &$\bf 1$ &$\bf 2$ \\
\hline
\ &\ &\ &\ &$7$ &$8$ &$2$ &$3$ &$\bf 9$ &$\bf 6$ &$\bf 4$ &$\bf 5$ &$\bf 1$ \\
\hline
\ &\ &\ &\ &$6$ &$7$ &$8$ &$9$ &$\bf 1$ &$\bf 2$ &$\bf 3$ &$\bf 4$ &$\bf 5$ \\
\hline
\end{tabular}
$$
$$
{\cal B}_3=
\vspace{.5cm}
\begin{tabular}
{|@{\hspace{6pt}}c@{\hspace{6pt}}
|@{\hspace{6pt}}c@{\hspace{6pt}}
|@{\hspace{6pt}}c@{\hspace{6pt}}
|@{\hspace{6pt}}c@{\hspace{6pt}}
||@{\hspace{2.3pt}}c@{\hspace{2.3pt}}
|@{\hspace{2.3pt}}c@{\hspace{2.3pt}}
|@{\hspace{2.3pt}}c@{\hspace{2.3pt}}
|@{\hspace{2.3pt}}c@{\hspace{2.3pt}}
||@{\hspace{2.5pt}}c@{\hspace{2.5pt}}
|@{\hspace{2.5pt}}c@{\hspace{2.5pt}}
|@{\hspace{2.5pt}}c@{\hspace{2.5pt}}
|@{\hspace{2.5pt}}c@{\hspace{2.5pt}}
|@{\hspace{2.5pt}}c@{\hspace{2.5pt}}|}
\ &\ &\ &\ &\ &\ &\ &\ &\ &\ &\ &\ &\ \\[-1.5 mm] \hline
\ &\ &\ &\ &\ &\ &\ &\ &\ &\ &\ &\ &\ \\ \hline
\ &\ &\ &\ &\ &\ &\ &\ &$\bf 4$ &$\bf 5$ &$\bf 1$ &$\bf 2$ &$\bf 3$ \\ \hline
\ &\ &\ &\ &$8$ &$9$ &$6$ &$7$ &$\bf 3$ &$\bf 4$ &$\bf 5$ &$\bf 1$ &$\bf 2$ \\
\hline
\ &\ &\ &\ &$7$ &$8$ &$2$ &$3$ &$\bf 9$ &$\bf 6$ &$\bf 4$ &$\bf 5$ &$\bf 1$ \\
\hline
\ &\ &\ &\ &$6$ &$7$ &$8$ &$9$ &$\bf 1$ &$\bf 2$ &$\bf 3$ &$\bf 4$ &$\bf 5$ \\
\hline
\ &\ &\ &\ &$9$ &$6$ &$7$ &$8$ &$\bf 5$ &$\bf 1$ &$\bf 2$ &$\bf 3$ &$\bf 4$ \\
\hline
\end{tabular}
%\hspace{.5cm}
\ {\cal B}_4=
\begin{tabular}
{|@{\hspace{6pt}}c@{\hspace{6pt}}
|@{\hspace{6pt}}c@{\hspace{6pt}}
|@{\hspace{6pt}}c@{\hspace{6pt}}
|@{\hspace{6pt}}c@{\hspace{6pt}}
||@{\hspace{2.3pt}}c@{\hspace{2.3pt}}
|@{\hspace{2.3pt}}c@{\hspace{2.3pt}}
|@{\hspace{2.3pt}}c@{\hspace{2.3pt}}
|@{\hspace{2.3pt}}c@{\hspace{2.3pt}}
||@{\hspace{2.5pt}}c@{\hspace{2.5pt}}
|@{\hspace{2.5pt}}c@{\hspace{2.5pt}}
|@{\hspace{2.5pt}}c@{\hspace{2.5pt}}
|@{\hspace{2.5pt}}c@{\hspace{2.5pt}}
|@{\hspace{2.5pt}}c@{\hspace{2.5pt}}|}
\ &\ &\ &\ &\ &\ &\ &\ &\ &\ &\ &\ &\ \\[-1.5 mm] \hline
\ &\ &\ &\ &\ &\ &\ &\ &\ &\ &\ &\ &\ \\ \hline
\ &\ &\ &\ &\ &\ &\ &\ &$\bf 3$ &$\bf 4$ &$\bf 5$ &$\bf 1$ &$\bf 2$ \\ \hline
\ &\ &\ &\ &$7$ &$8$ &$2$ &$3$ &$\bf 9$ &$\bf 6$ &$\bf 4$ &$\bf 5$ &$\bf 1$ \\
\hline
\ &\ &\ &\ &$6$ &$7$ &$8$ &$9$ &$\bf 1$ &$\bf 2$ &$\bf 3$ &$\bf 4$ &$\bf 5$ \\
\hline
\ &\ &\ &\ &$9$ &$6$ &$7$ &$8$ &$\bf 5$ &$\bf 1$ &$\bf 2$ &$\bf 3$ &$\bf 4$ \\
\hline
\ &\ &\ &\ &$8$ &$9$ &$6$ &$7$ &$\bf 4$ &$\bf 5$ &$\bf 1$ &$\bf 2$ &$\bf 3$ \\
\hline
\end{tabular}
$$
}\end{example}
\begin{technique}{{\bf [Gear technique 2]}
\label{tech4}
%technique 4 in Hadi's thesis
\noindent Similar to Technique \ref{tech3}, we take a completable partial latin square of order $n$
 (the following figure) where $A=\{1,\ldots,a\}$ and $B=\{a+1,\ldots,a+b\}$.
Then by permuting the rows of the two subrectangles of order $a\times (y+x)$
and $b\times y$, we
obtain a $\mu$-way latin square~of order $n$ with $n^2-(a(x+y)+by)$ fixed cells.\\
%-------Figure 2.10 from Hadi's Thesis--------\\
%
\def\sp{\hspace{1.2cm}}
\def\spp{\hspace{2.4cm}}
\def\d{\displaystyle}
\def\arraystretch{1}
$$
\begin{array}{|c|c|c|cc}
\cline{1-3}
\multicolumn{3}{|c|}{}\\[3mm]
\multicolumn{2}{|c}{}&\multicolumn{1}{r|}{\hspace{-1.5mm}\overbrace{\spp}^{{\d a}}\hspace{-1.5mm}}\\
\cline{3-3}
\multicolumn{2}{|c|}{}&\multicolumn{1}{|c|}{}&
\hspace{-11mm}\left.\begin{array}{c}
\\[.5cm]
\end{array}\right\}x\\[-7.5mm]
\multicolumn{1}{|c}{\sp}&\multicolumn{1}{c}{\hspace{-1.5mm}\overbrace{\sp}^{{\d b}}\hspace{-1.5mm}}&\multicolumn{1}{|c|}{}\\
\cline{2-2}
&&&
\hspace{-4mm}\left.\begin{array}{c}
\\[1cm]
\end{array}\right\}\mu-1\\
\cline{2-3}
&\multicolumn{2}{|c|}{}\\[6mm]
\cline{1-3}
\multicolumn{4}{c}{}\\[-6mm]
\multicolumn{1}{c}{\underbrace{\spp}_{\d c}}\\[-31.5mm]
\multicolumn{4}{c}{}&
\hspace{-5mm}\left.\begin{array}{c}
\\[2.2cm]
\end{array}\right\}y\\[-1.8cm]
\multicolumn{1}{c}{}&\multicolumn{1}{c}{B}\\[-1.1cm]
\multicolumn{1}{c}{}&\multicolumn{1}{r}{}&\multicolumn{1}{c}{A}\\[1.2cm]
\multicolumn{1}{c}{}&\multicolumn{2}{c}{A\cup B}
\end{array}
$$\\
%\

\vskip 8mm
Parallel to Technique \ref{tech3}, sufficient conditions for having a completable partial latin square suitable for this technique are the following.
\begin{enumerate}{
 \item[(1)] $a,b\geq x+\mu-1$,
 \item[(2)] $a+b+c=n$,
 \item[(3)] $x\geq 1$,
 \item[(4)] $c \geq y\geq \mu$,
 \item[(5)] $a+b\geq x+y$ and
}\end{enumerate}
Note that since the elements appearing in the bottom-left subrectangle of order $c \times y$ cannot be from the set $A \cup B$ we should have $c \geq y$ (see (4) above) for the partial latin square to be completable. 
}\end{technique}

%%%%%%%%%%%%%%%%%%%%%%%%%%%%%%%%%

%%%%%%%%%%%%%%%%%%%%%%%%%%%%%%%%%%%%%%%%%%%%%%%%%%%%%%%%%%%%%%%%%%%%%%%%%%%%
\section{Main results}
\label{sec: main-results}
%\setcounter{theorem}{0}
%\setcounter{preproposition}{0}
%\setcounter{pretechnique}{0}
%\setcounter{precorollary}{0}
%\setcounter{prelemma}{0}
%\setcounter{preexample}{0}

%%%%%%%%%%%%%%%%%%%%%%%%%%%%%%%%%%%%%%%%%%%%%%%%%%%%%%

In this section first we introduce some notations. Then we give some proofs of the existence and non-existence of our results on $4$-way intersection problem. These will lead to the proof of our main theorem stated at the end of this section.
\\

For each $n\geq 1$, define
$$J^4[n]=[0,n^2-27] \cup \{n^2-25,n^2-24,n^2-23,n^2-20,n^2-16,n^2\}.$$
%$$= [0,n^2] \setminus ([1,15]\cup \{17,18,19,21,22,26\}).$$

Since the possible volumes for $4$-way latin trades
are $N\setminus([1,15]\cup \{17,18,19,21,22,26\})$ (see \cite{MR1904722}) we have $I^4[n]\subseteq J^4[n]$.

The following lemmata are generalizations of Corollary 2.1 and Corollary 2.2 of \cite{MR1874724}:

\begin{lemma}
\label{lemma: J: n->2n}
If $I^4[n] \supseteq [0 , \lceil n^2/2 \rceil] $, then
$I^4[2n] \supseteq   [0 , 3n^2 ] \cup \{I^4[n] + \{3n^2\}\}.$
\end{lemma}

\begin{lemma}
\label{lemma: J: n->2n+1}
If $n \ge 4$ and
$I^4[n] \supseteq [0 , 7n+4] $, then
$$I^4[2n+1]   \supseteq
[0 , (2n+1)^2-n^2 ] \cup \{I^4[n] + \{(2n+1)^2-n^2\}\} .$$
\end{lemma}

Now, these two corollaries are immediate:

\begin{corollary}{
\label{corollary: J: n->2n}
If $I^4[n] = J^4[n]$ and
$I^4[n] \supseteq [0 , \lceil n^2/2 \rceil] $, then
$I^4[2n] = J^4[2n].$
}\end{corollary}

\begin{corollary}{
\label{corollary: J: n->2n+1}
If $n \ge 4$,  $I^4[n] = J^4[n]$ and
$I^4[n] \supseteq [0 , 7n+4] $, then
$I^4[2n+1] = J^4[2n+1].$
}\end{corollary}

%%%%%%%%%%%%%%%%%%%%%%%%%%%%%%%%%%%%%%%%%%%%%%%%%%%%%%
\subsection{Proofs of existence} \label{proofs of existence}
%writing proofs for every single case, as in Thesis's.
Here, we mention the method of obtaining non-trivial
members of $I^4[n]$, for each $n \geq 4$.

\begin{itemize}
\item $n=4$:
\begin{itemize}
\item[$\diamond$] $\{0\}$: by Lemma~\ref{lemma2.2.2 Hadi Thesis}.
\end{itemize}

\item $n=5$:
\begin{itemize}
\item[$\diamond$] $\{0,5\}$: by Lemma~\ref{lemma2.2.2 Hadi Thesis}
\item[$\diamond$] $\{1\}$: computer search.
\end{itemize}

\item $n=6$:
\begin{itemize}
\item[$\diamond$] $\{0,6,12\}$: by Lemma~\ref{lemma2.2.2 Hadi Thesis}
\item[$\diamond$] $\{1,2,3,4,5,7,8,11\}$: computer search.
\end{itemize}

\item $n=7$:
\begin{itemize}
\item[$\diamond$] $\{0,7,14,21\}$: by Lemma~\ref{lemma2.2.2 Hadi Thesis}
\item[$\diamond$] $\{1,2,3,4,5,6,8,9,10,11,12,13,15,16,17,19\}$: computer search.
\end{itemize}

\item $n=8$:
\begin{itemize}
\item[$\diamond$] $\{0,8,16,24,32\}:$ by Lemma~\ref{lemma2.2.2 Hadi Thesis}
\item[$\diamond$] $\{1,2,3,4,5,6,7,9,10,11,12,13,14,15,17,18,19,20,21,22,23,25,28,33\}$: computer search
\item[$\diamond$] $\{36\}$: Example \ref{example2.3.4 Hadi Thesis}
\item[$\diamond$] $\{48\}$: by Proposition \ref{theorem2.3.1 Hadi Thesis}.
\end{itemize}

\item $n=9$:
\begin{itemize}
\item[$\diamond$] $\{0,9,18,27,36,45\}$: by Lemma~\ref{lemma2.2.2 Hadi Thesis}
\item[$\diamond$] $\{2,3,4,6,7,8,10,11,12,13,14,15\}$: computer search
\item[$\diamond$] $\{1,5,16,17,20,21,25,29,37,41,61,65\}$: by Technique \ref{tech1}
\item[$\diamond$] $\{22,23,31,32,40\}$: by Technique \ref{tech3}.
\end{itemize}

\item $n=10$:
\begin{itemize}
\item[$\diamond$] $\{0,10,20,30,40,50,60\}$: by Lemma~\ref{lemma2.2.2 Hadi Thesis}
\item[$\diamond$] $\{1,2,3,4,5,6,7,8,11,12,15,16,25,26,27,28,31,32,35\}\cup$
\par\noindent $\{36,51,52,55,56,57,61,75,76,80,84\}$: by Proposition \ref{theorem2.3.1 Hadi Thesis}
\item[$\diamond$] $\{14,24,34,44,45,46,54\}$: by Technique \ref{tech3}.
\item[$\diamond$] $\{77\}$: computer search
\end{itemize}

\item $n=11$:
\begin{itemize}
\item[$\diamond$] $\{4,21,38,45,50,51,54,59,60,62,65,70\}$: by Technique \ref{tech3}
\item[$\diamond$] remaining elements of $I^4[11] \setminus R^4[11]:$ by Technique \ref{tech1}.
\end{itemize}

\item $n=12$:
\begin{itemize}
\item[$\diamond$] $\{117,121\}$: computer search
\item[$\diamond$] remaining elements of $I^4[12] \setminus R^4[12]:$ by Proposition \ref{theorem2.3.1 Hadi Thesis}.
\end{itemize}

\item $n=13$:
\begin{itemize}
\item[$\diamond$] $\{146\}$: computer search
\item[$\diamond$] $\{128,129\}$: by Technique \ref{tech4}
\item[$\diamond$] remaining elements of $I^4[13] \setminus R^4[13]:$ by Technique \ref{tech1}.
\end{itemize}

\item $n=14$:
\begin{itemize}
\item[$\diamond$] $\{169,173\}$: computer search
\item[$\diamond$] $\{135\}$: by Technique \ref{tech3}
\item[$\diamond$] $\{146\}$: by Technique \ref{tech4}
\item[$\diamond$] remaining elements of $I^4[14] \setminus R^4[14]:$ by Proposition \ref{theorem2.3.1 Hadi Thesis}.
\end{itemize}

\item $n=15$:
\begin{itemize}
\item[$\diamond$] $\{202\}$: by adding fixed cells, one can obtain a $4$-way latin rectangle of $5\times 15$ from the
the $4$-way latin rectangle found for the case $77\in I^4[10]$ in the Appendix. Then, obviously, adding
fixed rows to this $4$-way rectangle result in $202\in I^4[15]$.
\item[$\diamond$] $\{198\}$: computer search
\item[$\diamond$] $\{160\}$: by Technique \ref{tech3}
\item[$\diamond$] $\{170,171,173,174,175\}$: by Technique \ref{tech4}
\item[$\diamond$] remaining elements of $I^4[15] \setminus R^4[15]:$ by Technique \ref{tech1}.
\end{itemize}

\item $n=16$:
\begin{itemize}
\item[$\diamond$] $\{233\}$: similar to $202\in I^4[15]$ argument
\item[$\diamond$] $\{229\}$: computer search
\item[$\diamond$] $I^4[16]:$ by Proposition \ref{theorem2.3.1 Hadi Thesis}.
\end{itemize}

\item $n=17$:
\begin{itemize}
\item[$\diamond$] $\{266\}$: similar to $202\in I^4[15]$ argument
\item[$\diamond$] $\{262\}$: by adding fixed cells, one can obtain a $4$-way latin rectangle of $5\times 17$ from the
the $4$-way latin rectangle found for the case $117\in I^4[12]$ in the Appendix. Then, obviously, adding
fixed rows to this $4$-way rectangle result in $262\in I^4[17]$.
\item[$\diamond$] $\{215\}$: by Technique \ref{tech3}
\item[$\diamond$] $\{218,223,224\}$: by Technique \ref{tech4}
\item[$\diamond$] $\{220\}$: Using the latin rectangle $(17,220)$ shown in Appendix as
input, we can make this intersection with a technique similar to
Technique \ref{tech4}. This latin rectangle has three row permuting
subrectangles. Note that the first row of one of these subrectangles
is separate from other rows of it.
\item[$\diamond$] remaining elements of $I^4[17]:$ by Technique \ref{tech1}.
\end{itemize}

\item $n=18$:
\begin{itemize}
\item[$\diamond$] $\{301\}$: similar to $202\in I^4[15]$ argument
\item[$\diamond$] $\{297\}$: similar to $262\in I^4[17]$ argument
\item[$\diamond$] remaining elements of $I^4[18]:$ by Proposition \ref{theorem2.3.1 Hadi Thesis}.
\end{itemize}

\item $n=19$:
\begin{itemize}
\item[$\diamond$] $\{338\}$: similar to $202\in I^4[15]$ argument
\item[$\diamond$] $\{334\}$: similar to $262\in I^4[17]$ argument
\item[$\diamond$] $\{264,273,274,278,279\}$: by Technique \ref{tech3}
\item[$\diamond$] $\{268\}$: Using the latin rectangle $(19,268)$ shown in Appendix as
input, we can make this intersection with a technique similar to
Technique \ref{tech4}. This latin rectangle has three row permuting
subrectangles.
\item[$\diamond$] remaining elements of $I^4[17]$ by Technique \ref{tech1}.
\end{itemize}

\item $n=20$:
\begin{itemize}
\item[$\diamond$] $\{373\}$: similar to $262\in I^4[17]$ argument
\item[$\diamond$] remaining elements of $I^4[20]:$ by Proposition \ref{theorem2.3.1 Hadi Thesis}.
\end{itemize}

\item $n=21$:
\begin{itemize}
\item[$\diamond$] $\{414\}$: similar to $262\in I^4[17]$ argument
\item[$\diamond$] $\{354,358\}$: computer search
\item[$\diamond$] remaining elements of $I^4[21]:$ by Technique \ref{tech1}.
\end{itemize}

\item $n=22$:
\begin{itemize}
\item[$\diamond$] $\{457\}$: similar to $262\in I^4[17]$ argument
\item[$\diamond$] remaining elements of $I^4[22]:$ by Proposition \ref{theorem2.3.1 Hadi Thesis}.
\end{itemize}

\item $n=23$:
\begin{itemize}
\item[$\diamond$] $\{502\}$: similar to $262\in I^4[17]$ argument
\item[$\diamond$] remaining elements of $I^4[23]:$ by Technique \ref{tech1}.
\end{itemize}

\item $24 \leq n \leq 31$:
\begin{itemize}
\item[$\diamond$] for odd $n$, $I^4[n]:$ by Technique \ref{tech1}
\item[$\diamond$] for even $n$, $I^4[n]:$ by Proposition \ref{theorem2.3.1 Hadi Thesis}.
\end{itemize}

\item $n \geq 32$:
\begin{itemize}
\item[$\diamond$] for even $n$, $I^4[n]:$ by Corollary \ref{corollary: J: n->2n}.
\item[$\diamond$] for odd $n$, $I^4[n]:$ by Corollary \ref{corollary: J: n->2n+1}.
\end{itemize}

\end{itemize}

%%%%%%%%%%%%%%%%%%%%%%%%%%%%%%%%%%%%%%%%%%%%%%%%%%%%%%
\subsection{Proofs of non-existence} \label{proofs of non-existence}
Here, we prove that certain intersection sizes are not possible for small
values of $n$. In most of the proofs,
we assume that a $4$-way latin square of corresponding intersection size exists.
Then by argument on the extendibility and completability of 'its trade'
to a $4$-way latin square of needed order, we reach a contradiction.

\begin {itemize}{
\item {\bf Proof of $2 \notin I^4[5]$}: 
Suppose $2 \in I^4[5]$. By Lemma~\ref{lemma2.2.3 Hadi Thesis}, we know that
each row (and column) has at most $1$ fixed cell. By a permutation on
rows and columns, we may assume that $(1,1)$ and $(2,2)$ are the fixed
cells. If these fixed cells have different elements, then there are
only $3$ elements left for the trade at $(1,2)$. Hence, they must have the
same element, say $1$. Now, it is easy to verify that there is not
enough room for $1$ in the third row as $1$ cannot appear in $(3,1)$ and $(3,2)$.
%Suppose $2 \in I^4[5]$. Its trade is
%$5$ by $5$ and there are two empty cells which are not in a common row
%or column. Assume cells $(1,1)$ and $(2,2)$ are empty cells.
%By Corollary~ \ref{corollary2.2.2 Hadi Thesis} we have $R_1=C_2$ and
%$R_2=C_1$. Hence, this trade is not completable. Contradiction!
%
\item {\bf Proof of $9 \notin I^4[5]$ and $20 \notin I^4[6]$}: 
In each case, the volume of the trade is $16$ and since there are
at least four filled cells in each row and in each column of the trade,
and also since each element appears at least $4$ times in the trade, 
that is indeed a $4$-way latin square of order $4$. 
But, this trade cannot be extended to a $4$-way latin square of
order $5$ or $6$, as necessary condition for embedding a latin square of
order $m$ in a latin square of order $n$ is that $n\geq 2m$ or $m = 2n$.
\item {\bf Proof of $13,16 \notin I^4[6]$ and $22,24,25,26,29,33 \notin I^4[7]$}:
For these intersections, we produced all possible row sequences. Then, according to
Lemmata~\ref{lemma: sequence 4-7} and \ref{lemma: sequence 4-6},
%and Corollary \ref{corollary: sequence min}
we ruled out all of them.\\
(Note that we can similarly prove $9 \notin I^4[5]$ and $20 \notin I^4[6]$ by
means of Lemmata~\ref{lemma: sequence 4-6} and \ref{lemma: sequence 4-5}).
\item {\bf Proof of $9 \notin I^4[6]$}: Suppose $9 \in I^4[6]$. There
are at least four filled cells in each row and column of its trade.
So, its trade is $5\times 6$ or $6\times 6$.\\
$5\times 6$: There are two possible skeletons (cells containing
a dot correspond to empty cells of trade) for this trade. The first
skeleton is as below.\\
\def\arraystretch{.9}
\begin{center}
\hspace{1.5cm}
\begin{tabular}
{|@{\hspace{3pt}}c@{\hspace{3pt}}
|@{\hspace{3pt}}c@{\hspace{3pt}}
|@{\hspace{3pt}}c@{\hspace{3pt}}
|@{\hspace{6pt}}c@{\hspace{6pt}}
|@{\hspace{6pt}}c@{\hspace{6pt}}
|@{\hspace{6pt}}c@{\hspace{6pt}}|} \hline
$\bullet$ &\ &\ &\ &\ &\ \\ \hline
\ &$\bullet$ &\ &\ &\ &\ \\ \hline
\ &\ &$\bullet$ &\ &\ &\ \\ \hline
\ &\ &\ &\ &\ &\ \\ \hline
\ &\ &\ &\ &\ &\ \\ \hline
\end{tabular}
%\begin{figure}[ht]
%\label{9 in I[6] 56 halate avval1} \vspace*{0mm} \caption{your figure 2-11}
%\end{figure}
\end{center}
As $|X|\leq 6$ (since there is a row of six nonempty cells $|X|=6$)
and each element of $X$ should appear in at least
four rows we have either $|R_i\cap R_j|=4$ or $|R_i\cap R_j|=5$ for distinct $i,j \in \{1,2,3\}$. First we show $|R_i\cap R_j|=5$ cannot happen for any distinct $i,j \in \{1,2,3\}$. Suppose $|R_1\cap R_2|=5$. This means $R_1=R_2$. On the other hand by Corollary~\ref{corollary2.2.3 Hadi Thesis} we have $R_1 = C_2 \cup C_3$ and $|C_2 \cap C_3| = 3$ and similarly $R_2 = C_1 \cup C_3$ and $|C_1 \cap C_3| = 3$. Hence $R_1 = C_2 \cup C_3 = C_1 \cup C_3 = R_2$ which results to $C_1 = C_2$. But according to Corollary~\ref{corollary2.2.3 Hadi Thesis} with respect to $C_1$, $C_2$ and $R_3$ we have $|C_1 \cap C_2| = 3$ which is a contradiction. Therefore $|R_i\cap R_j|=4$ for distinct $i,j \in \{1,2,3\}$. Now consider
$x\in X\setminus R_3$. Clearly $x\in R_1$ and $x\in R_2$ (hence $x \in R_1\cap R_2$).
By Lemma~\ref{lemma2.2.3 Hadi Thesis} we have $C_3\subset R_1$ and
$C_3\subset R_2$ (hence $C_3 \subseteq R_1\cap R_2$). Since $|C_3|=4$ we have
$x\in C_3 = R_1\cap R_2$. This yields that cell $(3,3)$ of this trade
cannot be filled to get a $4$-way latin square of needed order.\\
The second possible skeleton is\\
\def\arraystretch{.9}
\begin{center}
\hspace{1.5cm}
\begin{tabular}
{|@{\hspace{3pt}}c@{\hspace{3pt}}
|@{\hspace{3pt}}c@{\hspace{3pt}}
|@{\hspace{3pt}}c@{\hspace{3pt}}
|@{\hspace{6pt}}c@{\hspace{6pt}}
|@{\hspace{6pt}}c@{\hspace{6pt}}
|@{\hspace{6pt}}c@{\hspace{6pt}}|} \hline
$\bullet$ &$\bullet$ &\ &\ &\ &\ \\ \hline
\ &\ &$\bullet$ &\ &\ &\ \\ \hline
\ &\ &\ &\ &\ &\ \\ \hline
\ &\ &\ &\ &\ &\ \\ \hline
\ &\ &\ &\ &\ &\ \\ \hline
\end{tabular}
%\begin{figure}[ht]
%\label{9 in I[6] 56 halate avval2} \vspace*{0mm} \caption{�$X_2$�ڐ������}
%\end{figure}
\end{center}
As we have a row of six nonempty cells we have $X=\{1,2,\ldots,6\}$.
Suppose $6\notin X\setminus R_2$. To be able to fill the cell $(2,3)$
in order to obtain the needed $4$-way latin square, we should have $6\notin C_1,C_2,C_3$. But
this is a contradiction as $6$ should appear in at least $4$ columns of
this trade.\\
$6\times 6$: In this trade, there are at least three rows and three columns
with two empty cells. We can assume the first three rows and columns are so.
There is one row and one column with only one empty cell. By permutation,
assume that the last row and column are like this and hence cell $(6,6)$ is
empty. Suppose $R_6=\{1,2,\ldots,5\}$. We can assume, by
Corollary \ref{corollary2.2.3 Hadi Thesis}, that $C_1=\{1,2,4,5\},
C_2=\{1,2,3,4\}$ and $C_3=\{1,2,3,5\}$. According to
Corollary \ref{corollary2.2.2 Hadi Thesis}, we can assume that
$R_i=C_i$ for $i=1,2,3$ and hence by Corollary \ref{corollary2.2.3 Hadi Thesis}
we have $C_6=\{1,2,\ldots,5\}$. Each element should appear in at least four
columns, so $X=\{1,2,\ldots,5\}$. Now one can simply check that empty cells
of three first rows and columns cannot be filled in a latin way to obtain the
$4$-way latin square of needed order.
\item {\bf Proof of $20 \notin I^4[7]$}: 
Let ${\cal L}$ be a $4$-way latin square with intersection size $20$.
The row (or column) sequence of ${\cal L}$ can only contain $\{0, 1, 2, 3, 7\}$
(Lemma \ref{lemma:possible-in-seq}).
The row (or column) sequence cannot contain more than two $7$s as there
are only $20$ cells in the intersection part.
If it contains only one $7$, it cannot contain $3$ and
if it contains exactly two $7$s, it cannot contain $2$ or $3$ (Lemma \ref{lemma: sequence 4-7}).
In both cases, the maximum intersection size would be $19$ ($7 + 6 \times 2$ or $7 + 7 + 5 \times 1$).
Hence, the row sequence of ${\cal L}$ does not contain $7$.
Similarly, there can be no $1$s and there is exactly one $2$, and (up to symmetry)
the only valid row (and column) sequence is $\{a_i\} = \{3,3,3,3,3,3,2\}$. 
Without loss of generality, suppose that the first column
has intersection size $2$ and these intersections are the two top
cells of this column. At least one of the two first rows has
intersection size $3$. Suppose this row is the first one.
Furthermore, suppose that these $3$ intersections are located in
the first $3$ cells of this row. Applying Corollary
\ref{corollary2.2.2 Hadi Thesis}, we deduce that the same set of
numbers (Suppose $\{1,2,3\}$) appears in the intersection part in the last
four columns. So, there are $3$ numbers that appear at least four
times in ${\cal L}$. By Lemma \ref{lemma:possible-in-seq},
any of these numbers appears exactly $7$ times in the intersection part.
Hence, ${\cal L}$ has intersection size at least $7 \times
3 = 21 > 20$.

\item {\bf Proof of $37 \notin I^4[8], 54\notin I^4[9]$ and $73\notin I^4[10]$}: We prove
$54 \notin I^4[9]$ and the other ones are similar. Suppose $54 \in I^4[9]$. There
are at least four filled cells in each row and column of its trade. So, its trade is $5\times 6$
or $6\times 6$.\\
$5\times 6$: Each element of this trade should appear in at least four cells of it. So at most
six elements are in this trade. Since there is a row of six nonempty cells, we have
$X=\{1,2,\ldots,6\}$. To extend this trade to a $4$-way latin square of order 9, we should place
other three elements, $\{7,8,9\}$, in the added rows and empty cells. But this is
not possible.\\
$6\times 6$: The same as $6\times 6$ case of $9 \notin I^4[6]$ we can prove that $X=\{1,2,\ldots,5\}$.
But this trade cannot be extended to the desired $4$-way latin square using other four elements, $\{6,7,8,9\}$.
\item {\bf Proof of $39 \notin I^4[8]$ and $56\notin I^4[9]$}: We prove
$39 \notin I^4[8]$ and the other one is similar. Suppose $39 \in I^4[8]$. There
are at least four filled cells in each row and column of its trade. So, its
trade is $5\times 5$ or $5\times 6$ or $6\times 6$.\\
$5\times 5$: In this trade we have $|X|\leq 6$. So it cannot be extended
to a $4$-way latin square of order $8$.\\
$5\times 6$: In this trade we have $|X|\leq 6$ (one can show $|X|=5$). So it cannot be extended
to a $4$-way latin square of order $8$.\\
$6\times 6$: In this trade, there are five rows and five columns containing two empty
cells. There are one row and one column with a single empty cell. We can assume the last row
and column contain a single empty cell. If we show that $R_1=R_2=\cdots=R_5=C_1=C_2=\cdots=C_5$ then
we have a contradiction since $|X|\geq 5$ (last row has five nonempty cells).\\
Let $H=(A,B)$ be a bipartite graph constructed from this trade, as follows. Corresponding to each first five rows
we have a vertex in $A$ and corresponding to each first five columns we have a vertex in $B$.
Vertex $a\in A$ is adjacent to vertex $b\in B$ if and only if their correspondent rows and columns
have a nonempty cell in common.
%Let $H$ be a graph constructed as follows. Corresponding to each nonempty cell of the first five rows
%and columns of this trade, $H$ has a vertex. Two vertices of $H$ are adjacent if and only if their
%correspondent cells are in a common row or column. $H$ is subgraph of $K_5\square k_5$. More pre
As $|R_i|=|C_i|=4$ for $i\in \{1,2,\ldots,5\}$,
to prove $R_1=R_2=\cdots=R_5=C_1=C_2=\cdots=C_5$ it suffices to show the graph $H$ is connected.
To show this, consider that $H$ is a graph obtained from a $K_{5,5}$ by deleting
two perfect matchings from it. So, $H$ is a $3$-regular bipartite graph with five
vertices in each partition and hence connected.
%A possible skeleton of this trade is as below
%\def\arraystretch{.9}
%\begin{center}
%\begin{tabular}
%{|@{\hspace{3pt}}c@{\hspace{3pt}}
%|@{\hspace{3pt}}c@{\hspace{3pt}}
%|@{\hspace{3pt}}c@{\hspace{3pt}}
%|@{\hspace{3pt}}c@{\hspace{3pt}}
%|@{\hspace{3pt}}c@{\hspace{3pt}}
%||@{\hspace{3pt}}c@{\hspace{3pt}}|} \hline
%$\bullet$ &$\bullet$ &\ &\ &\ &\ \\ \hline
%$\bullet$ &$\bullet$ &\ &\ &\ &\ \\ \hline
%\ &\ &$\bullet$ &$\bullet$ &\ &\ \\ \hline
%\ &\ &\ &$\bullet$ &$\bullet$ &\ \\ \hline
%\ &\ &$\bullet$ &\ &$\bullet$ &\ \\ \hline \hline
%\ &\ &\ &\ &\ &$\bullet$ \\ \hline
%\end{tabular}
%\end{center}
%
\item {\bf Proof of $40 \notin I^4[8]$ and $57\notin I^4[9]$}: We prove
$40 \notin I^4[8]$ and the other one is similar. Suppose $40 \in I^4[8]$. There
are at least four filled cells in each row and column of its trade. So, its
trade is $5\times 5$ or $4\times 6$ or $5\times 6$ or $6\times 6$.\\
$5\times 5$: Each element of this trade should appear in at least four cells of it. So at most
six elements are in this trade. But this trade cannot be extended to the desired $4$-way latin square of
order $8$.\\
$4\times 6$: Since each element should appear in at least four cells of this trade we have
$X=\{1,2,\ldots,6\}$. So this trade cannot be extended to the desired $4$-way latin square of
order $8$ using the remaining two elements.\\
$5\times 6$: This trade has at least one row with two empty cells. Suppose the first row
has two empty cells. Each column has an empty cell. We can assume the last
two cells of first row are empty. So by Corollary \ref{corollary2.2.2 Hadi Thesis}
we have $R_1=C_1=C_2=C_3=C_4$. This yields that $X=R_1$. Therefore, this trade
cannot be extended to the desired $4$-way latin square of order $8$ using the remaining four elements.\\
$6\times 6$: Each row and column has two nonempty cells. Similar to $6\times 6$ case of
$39 \notin I^4[8]$ we can show that $X=\{1,2,3,4\}$. And hence this trade
cannot be extended to the desired $4$-way latin square of order $8$ using the remaining four elements.
\item {\bf Proof of $41 \notin I^4[8]$ and $58\notin I^4[9]$}: We prove
$41 \notin I^4[8]$ and the other one is similar. Suppose $41 \in I^4[8]$. There
are at least four filled cells in each row and column of its trade. So, its
trade is $5\times 5$. It has two empty cells which are not in a common row or column.
As $|X|=5$, this trade is not extendible to the desired $4$-way latin square.
\item {\bf Proof of $44 \notin I^4[8]$}: There are at least four filled
cells in each row and column of its trade. So, its trade is $5\times 5$ or $4\times 5$.\\
$4\times 5$: One can simply see that $|X|=5$. So, this trade is not extendible to the desired $4$-way latin square.\\
$5\times 5$: Each row and column contains an empty cell. Similar to $6\times 6$ case of
$39 \notin I^4[8]$ we can show that $X=\{1,2,3,4\}$. Hence this trade is not extendible
to the desired $4$-way latin square.
}\end{itemize}

The results given above follow  the following (main) theorem. Note that 
$R^4[n]$ denotes undecided values.
%, $J^4[n]\setminus I'^4[n]$.

%%%%%%%%%%%%%%%%%%%%%%%%%%%%%%%%%%%%%%%%%%%%%%%%%%%%%%

\begin{theorem} {\rm\bf (MAIN)}
We have,

\begin{itemize}
\item $I^4[n] = J^4[n]$ for $1 \le   n   \leq 6$,

%\item $I^4[1] = J^4[1] = \{1\}$,
%
%\item $I^4[2] = J^4[2] = \{4\}$,
%
%\item $I^4[3] = J^4[3] = \{9\}$,
%
%\item $I^4[4] = J^4[4] = \{0,16\}$,
%
%\item $I^4[5] = J^4[5] = \{0,1,5,25\}$,
%
%\item $I^4[6] = J^4[6] = [0,8] \cup \{11,12,36\}$,

\item $I^4[7] \supseteq [0,17] \cup \{19,21,49\}$,
\noindent $R^4[7]=\{18\}$,

\item $I^4[8] \supseteq J^4[8]\setminus (R^4[8]\cup \{35,37,39,40,41,44\})$,
\noindent $R^4[8]=\{26,27,29,30,31,34\}$

\item $I^4[9] \supseteq J^4[9]\setminus (R^4[9]\cup \{54,56,57,58\})$,\\
\noindent $R^4[9]=\{19,24,28,33,34,35,39,42,43,44,47,48,49,50,51,52,53\}$,

\item $I^4[10] \supseteq J^4[10]\setminus (R^4[10]\cup \{73\})$,
\noindent  $R^4[10]=\{9,13,17,18,19,21,22,23,29,33,37,38,39,41,42,$
\noindent $43,47,48,49,53,57,58,59,61,62,63,64,65,66,67,68,69,70,71,72\}$,

\item $I^4[11] \supseteq J^4[11]\setminus R^4[11]$,\\
\noindent $R^4[11]=\{74,75,78,80,81,82,83,84,85,86,87,88,89,90,92,93,94,98\}$,

\item $I^4[12] \supseteq J^4[12]\setminus R^4[12]$,
\noindent $R^4[12]=\{93,97,98,99,100,101,102,103,104,105,106,107\}$,

\item $I^4[13] \supseteq J^4[13]\setminus R^4[13]$,
\noindent $R^4[13]=\{118,119,121,122,123,124,125,126,130,131,132,142\}$,

\item $I^4[14] \supseteq J^4[14]\setminus R^4[14]$,
\noindent $R^4[14]=\{137,139,141,142,143,144,145\}$,

\item $I^4[15] \supseteq J^4[15]\setminus R^4[15]$,
\noindent $R^4[15]=\{162,164,166,167,168,172\}$, and

\item $I^4[n] = J^4[n]$ for any $n\geq 16$.
\end{itemize}
\end{theorem}
%%%%%%%%%%%%%%%%%%%%%%%%%%%%%%%%%%%%%%%
%%%%%%%%%%%%%%%%%%%%%%%%%%%%%%%%%%%%%%%
%%%%%%%%%%%%%%%%%%%%%%%%%%%%%%%%%%%%%%%
\section{Conclusion}
The so-called intersection problem has been considered for many different combinatorial structures,
including latin squares. This intersection problem basically takes a pair of structures,
with the same parameters and based on the same underlying set, and determines the possible
number of common sub-objects which they may have (such as blocks, entries, etc.).
The intersection problem has also been extended from consideration of pairs of
combinatorial structures to sets of three, or even sets of $\mu$, where $\mu$ may be
larger than $3$. In this paper, we studied the problem of determining,
for all orders $n$, the set of integers $k$ for which there exists $4$ latin squares
of order $n$ having precisely $k$ identical cells, with their remaining $n^2 - k$ cells
different in all four latin squares, denoted by $I^4[n]$.

We have completely determined $I^4[n]$ for $n \geq 16$ 
but there are still undecided values for $n \leq 15$.
The smallest undecided question is whether $18 \in I^4[7]$.
If the answer to this question is yes, it can be concluded that $137 \in I^4[14]$ and
$162 \in I^4[15]$ using Proposition \ref{theorem2.3.1 Hadi Thesis}
and Technique \ref{theorem2.3.2 Hadi Thesis}. All other undecided values 
($R^4[x], 8 \leq x \leq 15$, section \ref{sec: main-results})
should be answered directly and can not be solved
(at least with) techniques discussed in this paper.

%%%%%%%%%%%%%%%%%%%%%%%%%%%%%%%%%%%%%%%%%%%%%%%%%%%%%%%%%%%%%%%%%%%%%%%%%%%%
\section*{Acknowledgements}
Most of the computerized search for latin squares for this work were run on the computer site 
of the Department of the Mathematical Sciences, Sharif University of Technology, during holidays. 
We thank the department and specially Ms. Maryam Sadeghian, the site administrator, for her assistance. 

We are very grateful to anonymous referee for taking the time and effort necessary to
provide such insightful comments.

%%%%%%%%%%%%%%%%%%%%%%%%%%%%%%%%%%%%%%%%%%%%%%%%%%%%%%%%%%%%%%%%%%%%%%%%%%%%
%\newpage

\normalsize
\bibliographystyle{plain}
\bibliography{sr334MahNew}

\def\cprime{$'$}
\begin{thebibliography}{1}

\bibitem{MR1821945}
Peter Adams, Elizabeth~J. Billington, Darryn~E. Bryant, and A.~Khodkar.
\newblock The {$\mu$}-way intersection problem for {$m$}-cycle systems.
\newblock {\em Discrete Math.}, 231(1-3):27--56, 2001.
\newblock 17th British Combinatorial Conference (Canterbury, 1999).

\bibitem{MR1874724}
Peter Adams, Elizabeth~J. Billington, Darryn~E. Bryant, and E.~S. Mahmoodian.
\newblock The three-way intersection problem for {L}atin squares.
\newblock {\em Discrete Math.}, 243(1-3):1--19, 2002.

\bibitem{MR1904722}
Peter Adams, Elizabeth~J. Billington, Darryn~E. Bryant, and Ebadollah~S.
  Mahmoodian.
\newblock On the possible volumes of {$\mu$}-way {L}atin trades.
\newblock {\em Aequationes Math.}, 63(3):303--320, 2002.

\bibitem{MR1278954}
Elizabeth~J. Billington.
\newblock The intersection problem for combinatorial designs.
\newblock {\em Congr. Numer.}, 92:33--54, 1993.
\newblock Twenty-second Manitoba Conference on Numerical Mathematics and
  Computing (Winnipeg, MB, 1992).

\bibitem{MR2202131}
Yanxun Chang, Giovanni Lo~Faro, and Giorgio Nordo.
\newblock The fine structures of three {L}atin squares.
\newblock {\em J. Combin. Des.}, 14(2):85--110, 2006.

\bibitem{MR1093150}
Chin~Mei Fu and Hung-Lin Fu.
\newblock The intersection of three distinct {L}atin squares.
\newblock {\em Matematiche (Catania)}, 44(1):21--45 (1990), 1989.

\bibitem{MR1125351}
Chin~Mei Fu and Hung-Lin Fu.
\newblock The intersection problem of {L}atin squares.
\newblock {\em J. Combin. Inform. System Sci.}, 15(1-4):89--95, 1990.
\newblock Graphs, designs and combinatorial geometries (Catania, 1989).

\bibitem{Fu}
H-L.\ Fu.
\newblock {\em On the construction of certain type of {L}atin squares with
  prescribed intersections}.
\newblock PhD thesis, Auburn University, 1980.

\bibitem{MR941781}
Salvatore Milici and Gaetano Quattrocchi.
\newblock On the intersection problem for three {S}teiner triple systems.
\newblock In {\em Proceedings of the {F}irst {C}atania {I}nternational
  {C}ombinatorial {C}onference on {G}raphs, {S}teiner {S}ystems, and their
  {A}pplications, {V}ol.\ 1 ({C}atania, 1986)}, volume 24A, pages 175--194,
  1987.

\end{thebibliography}
\section*{Appendix \label{appendix}}
{ In each of the following, the top two numbers in each row
indicate the order of latin square and the achieved intersection
number, respectively. Elements of the set
$\{1,2,\ldots,9,a,b,c,\ldots,z\}$ are used
as entries.\\
\noindent Latin rectangles containing underlined entries, are the
input needed for Techniques \ref{tech3} and \ref{tech4} where the
underlined entries show fixed cells. Fixed or permuting parts are
separated by double lines.

\newpage

\textwidth = 15 cm
\textheight = 23 cm
\oddsidemargin = -2 cm
\evensidemargin = 0 cm
\topmargin = -1 cm
\parskip = 1.5 mm

\def\mfour#1#2#3#4{\raise .02ex\hbox{%
    $#1_{\displaystyle #2}#3_{\displaystyle #4}$}}

\def\mfour#1#2#3#4{\hbox{%         % You can change 1
     ${#1_{{\displaystyle #2}}#3_{{\displaystyle #4}}}$}}

\def\mfour#1#2#3#4{\raise 1ex\hbox{${#1_{{\displaystyle #2}_{{\displaystyle #3}_{\displaystyle #4}}}}$}}

\def\un{\underline}

\def\arraystretch{1.4}                 % You can change 1.8

\begin{center}

\footnotesize {



}

\end{center}
}
%
%\normalsize
%\bibliographystyle{plain}
%\bibliography{sr334Mah}
%%%%%%%%%%%%%%%%%%%%%%%%%%%%%%%%%%%%%%%%%%%%%%%%%%%%%%%%%%%%%%%%%%%%%%%%%%%%%
\end{document}